\documentclass[twoside,reqno,11pt]{amsart}

\usepackage{latexsym}

\newcommand{\qdn}{\hspace*{-1.5mm}}

\newcommand{\xqdn}{\hspace*{-5.0mm}}
\newcommand{\xxqdn}{\hspace*{-10mm}}

\newcommand{\fns}{\footnotesize}

\newcommand{\ffnk}[4]{\left[\qdn\ba{#1}#3\\[0.7mm]#4\ea{\!\Big|\:#2}\right]}

\newcommand{\binm}{\binom}

\newcommand{\nnm}{\nonumber}
\newcommand{\be}{\begin{equation}}
\newcommand{\ee}{\end{equation}}
\newcommand{\ba}{\begin{array}}
\newcommand{\ea}{\end{array}}
\newcommand{\bmn}{\begin{eqnarray}}
\newcommand{\emn}{\end{eqnarray}}
\newcommand{\bnm}{\begin{eqnarray*}}
\newcommand{\enm}{\end{eqnarray*}}
\newcommand{\bln}{\begin{subequations}}
\newcommand{\eln}{\end{subequations}}

\newtheorem{thm}{Theorem}

\newtheorem{corl}[thm]{Corollary}

\newtheorem{entry}{Entry}

\newcommand{\bbtm}[4]{\bibitem{kn:#1}{#2,}~{#3,}~{#4.}}
\newcommand{\cito}[1]{\cite{kn:#1}}
\newcommand{\citu}[2]{\cite[#2]{kn:#1}}

\begin{document} 
{\fns
\title{q-Analogues of several $\pi$-formulas}
\author{Chuanan Wei}
\dedicatory{
Department of Medical Informatics\\
  Hainan Medical University, Haikou 571199, China}
\thanks{\emph{Email address}:
      weichuanan78@163.com}

\address{ }


\begin{abstract}
According to the $q$-series method, a short proof for Hou and Sun's
identity, which is the $q$-analogue of a known $\pi$-formula, is
offered. Furthermore, $q$-analogues of several other $\pi$-formulas
are also established in terms of the $q$-series method and the
telescoping method.
\end{abstract}

\maketitle\thispagestyle{empty}
\markboth{Chuanan Wei}
         {q-Analogues of several $\pi$-formulas}

\section{Introduction}

For two complex numbers $x$, $q$ and an integer $n$ with $|q|<1$,
define $q$-shifted factorial to be
 \bnm
(x;q)_{\infty}=\prod_{i=0}^{\infty}(1-xq^i),\quad
(x;q)_n=\frac{(x;q)_{\infty}}{(xq^n;q)_{\infty}}.
 \enm
For convenience, we shall adopt the following notations:
 \bnm
&&(x_1,x_2,\cdots,x_r;q)_{\infty}=(x_1;q)_{\infty}(x_2;q)_{\infty}\cdots(x_r;q)_{\infty},\\
&&(x_1,x_2,\cdots,x_r;q)_{n}=(x_1;q)_{n}(x_2;q)_{n}\cdots(x_r;q)_{n}.
 \enm
Following Gasper and Rahman \cito{gasper}, the basic hypergeometric
series can be defined by
 \bnm
{_{1+r}\phi_s}\ffnk{cccccc}{q;z}{a_0,&a_1,&\cdots,a_r}{&b_1,&\cdots,b_s}
  =\sum_{k=0}^{\infty}\frac{(a_0,a_1,\cdots,a_r;q)_k}{(q,b_1,\cdots,b_s;q)_k}
\Big\{(-1)^kq^{\binm{k}{2}}\Big\}^{s-r}z^k.
 \enm

Recently, Guo and Liu \cito{guo-a} utilized WZ method to derive the
two identities
 \bmn
&&\sum_{k=0}^{\infty}q^{k^2}\frac{1-q^{6k+1}}{1-q}\frac{(q;q^2)_k^2(q^2;q^4)_k}{(q^4;q^4)_k^3}
=\frac{(1+q)(q^2;q^4)_{\infty}(q^6;q^4)_{\infty}}{(q^4;q^4)_{\infty}^2},
  \label{q-ramanujan-a}\\\label{q-ramanujan-b}
&&\sum_{k=0}^{\infty}(-1)^kq^{3k^2}\frac{1-q^{6k+1}}{1-q}\frac{(q;q^2)_k^3}{(q^4;q^4)_k^3}
=\frac{(q^3;q^4)_{\infty}(q^5;q^4)_{\infty}}{(q^4;q^4)_{\infty}^2},
 \emn
which are $q$-analogues of two formulas due to Ramanujan
\cito{ramanujan}
 \bnm
&&\sum_{k=0}^{\infty}(6k+1)\frac{(\frac{1}{2})_k^3}{k!^34^k}=\frac{4}{\pi},\\
&&\sum_{k=0}^{\infty}(-1)^k(6k+1)\frac{(\frac{1}{2})_k^3}{k!^38^k}=\frac{2\sqrt{2}}{\pi}.
\enm
 Here, we have employed the Pochhammer symbol
 \[(x)_0=1\quad\text{and}\quad(x)_n=x(x+1)\cdots(x+n-1)\quad\text{when}\quad n=1,2,\cdots.\]

 Hou, Krattenthaler, and Sun \cito{sun-a} showed that
 \eqref{q-ramanujan-a} and \eqref{q-ramanujan-b} can be deduced
 through the intelligent use of a quadratic transformation formula from
 \citu{gasper}{p. 92} and gave $q$-analogues of the
 classical Leibniz series and a Zeilberge-type series. Later, Guo
 and Zudilin \cito{guo-b} also proved the two formulas in accordance with a quadratic
 summation formula
 of Rahman \cito{rahman} and established $q$-analogues
 of several other  Ramanujan-type formulas via the WZ-method.
 Recall three simple $\pi$-formulas
 \bmn\label{pi-a}
&&\xxqdn\sum_{k=1}^{\infty}\frac{1}{k^2}=\frac{\pi^2}{6},\\
&&\xxqdn\sum_{k=0}^{\infty}\frac{(-1)^k}{(2k+1)^3}=\frac{\pi^3}{32},
  \label{pi-b}\\\label{pi-c}
&&\xxqdn\sum_{k=1}^{\infty}\frac{1}{k^4}=\frac{\pi^4}{90}.
 \emn
Sun \cito{sun-b} found two $q$-analogues of \eqref{pi-a} and a
$q$-analogue of \eqref{pi-c}. Hou and Sun \cito{sun-c} offered the
following $q$-analogue of \eqref{pi-b}:
  \bmn\label{sun}
\sum_{k=0}^{\infty}(-1)^k\frac{q^{2k}(1+q^{2k+1})}{(1-q^{2k+1})^3}
=\frac{(q^2;q^4)_{\infty}^2(q^4;q^4)_{\infty}^6}{(q;q^2)_{\infty}^4}.
 \emn
More $q$-analogues of $\pi$-formulas can be seen in the papers
\cito{guillera-c} and \cito{guo-c}.

 In the literature of $\pi$-formulas, there exist five interesting
 results
 \bmn
 &&\label{weisstein-a}\quad\qquad\sum_{k=0}^{\infty}\frac{k!}{(\frac{3}{2})_k}\frac{1}{2^k}=\frac{\pi}{2},\\
 &&\label{weisstein-b}\quad\qquad\sum_{k=0}^{\infty}\frac{k!}{(\frac{3}{2})_k}\frac{1}{4^k}=\frac{2\pi}{3\sqrt{3}},\\
 &&\label{guillera-a}\quad\qquad\sum_{k=0}^{\infty}\frac{k!^3}{(\frac{3}{2})_k^3}\frac{3k+2}{4^k}=\frac{\pi^2}{4},\\
 \frac{\sin(\pi x)\sin(\pi y)}{\pi^2}&&\xqdn\!=\sum_{k=0}^{\infty}\frac{(x)_k(1-x)_k(y)_k(1-y)_k}{k!^2(k+1)!^2}
 \nnm\\\label{wei-a}
 &&\xqdn\!\times\:\big\{(k^2+k)(x+y-x^2-y^2)+xy(1-x)(1-y)\big\},\\
 \frac{\pi^2}{\sin(\pi x)\sin(\pi y)}&&\!\xqdn=\frac{1}{(1-x)(1-y)}+\sum_{k=0}^{\infty}\frac{k!^4}{(x)_{k+1}(1-x)_{k+2}(y)_{k+1}(1-y)_{k+2}}
  \nnm\\\label{wei-b}
  &&\!\xqdn\times\:\big\{(k^2+2k)(2-2x-2y+x^2+y^2)+1-xy(2-x)(2-y)\big\},
 \emn
where \eqref{weisstein-a}-\eqref{weisstein-b} can be found in
Weisstein \cito{weisstein}, \eqref{guillera-a} can be seen in
Guillera \cito{guillera-a}, and \eqref{wei-a}-\eqref{wei-b} are due
to Wei et al. \cito{wei-a}. It should be mentioned that the case
$x=y=\frac{1}{2}$ of \eqref{wei-a} is exactly the result of Guillera
\cito{guillera-b}:
 \bmn \label{guillera-b}
\qquad
\frac{2}{\pi^2}=\sum_{k=0}^{\infty}\frac{(1/2)_{k}^4}{k!^2(k+1)!^2}\{k^2+k+1/8\}.
 \emn

The structure of the paper is arranged as follows. In Section 2, we
shall prove \eqref{sun}, construct a new $q$-analogue of
\eqref{pi-a}, and derive $q$-analogues of
\eqref{weisstein-a}-\eqref{guillera-a} by means of the $q$-series
method. $q$-Analogues of \eqref{wei-a}-\eqref{guillera-b} will be
established through the telescoping method in Section 3.

\section{$q$-Series method and $q$-analogues of $\pi$-formulas}
Above all, we shall give a simple proof of \eqref{sun} in this
section. Subsequently, a new $q$-analogue of \eqref{pi-a} will be
deduced in Theorem \ref{thm-b}. Then $q$-analogue of
\eqref{weisstein-a}-\eqref{guillera-a} will respectively be
established in Theorems \ref{thm-c}-\ref{thm-e}.

\begin{thm}\label{thm-a}
\bnm
\sum_{k=0}^{\infty}(-1)^k\frac{q^{2k}(1+q^{2k+1})}{(1-q^{2k+1})^3}
=\frac{(q^2;q^4)_{\infty}^2(q^4;q^4)_{\infty}^6}{(q;q^2)_{\infty}^4}.
 \enm
\end{thm}

\begin{proof}
We begin with the summation formula for ${_8\phi_7}$-series (cf.
\citu{gasper}{p. 355})
 \bnm
  &&\xqdn\xxqdn{_8\phi_7}\ffnk{ccccccc}{q;c}{-c,q\sqrt{-c},-q\sqrt{-c},a,q/a,c,-d,-q/d}{\sqrt{-c},-\sqrt{-c},-cq/a,-ac,-q,cq/d,cd}\\
&&\xqdn\xxqdn\:\:=\:\frac{(-c,-cq;q)_{\infty}(acd,acq/d,cdq/a,cq^2/ad;q^2)_{\infty}}{(cd,cq/d,-ac,-cq/a;q)_{\infty}}.
 \enm
Fix $a=q^{\frac{1}{2}}$, $c=-q$, $d=-q^{\frac{1}{2}}$ to obtain
 \bnm
  \qquad{_5\phi_4}\ffnk{ccccccc}{q;-q}{q,-q^{\frac{3}{2}},q^{\frac{1}{2}},q^{\frac{1}{2}},q^{\frac{1}{2}}}
  {-q^{\frac{1}{2}},q^{\frac{3}{2}},q^{\frac{3}{2}},q^{\frac{3}{2}} }
=\frac{(q,q^2;q)_{\infty}(q^2;q^2)_{\infty}^4}{(q^{\frac{3}{2}};q)_{\infty}^4}.
\enm
 Replacing $q$ by $q^2$ in the last equation, we get Theorem
 \ref{thm-a}.
\end{proof}

\begin{thm}\label{thm-b}
\bnm
\:\:\quad\sum_{k=0}^{\infty}\frac{1+q^{4k+2}}{(1+q^{2k+1})^2}\frac{q^{2k}}{(1-q^{2k+1})^2}
=\frac{(-q^2;q^2)_{\infty}^2(q^4;q^4)_{\infty}^4}{(q^2;q^4)_{\infty}^2}.
 \enm
\end{thm}

\begin{proof}
A transformation formula for $_8\phi_7$-series (cf. \citu{gasper}{p.
361}) can be stated as
  \bnm
  &&\xqdn{_8\phi_7}\ffnk{ccccccc}
 {q;\frac{a^2q^2}{bcdef}}{a,q\sqrt{a},-q\sqrt{a},b,c,d,e,f}{\sqrt{a},-\sqrt{a},aq/b,aq/c,aq/d,aq/e,aq/f}\\
 &&\xqdn\:\:=\:\frac{(aq,aq/ef,\lambda q/e,\lambda q/f;q)_{\infty}}{(aq/e,aq/f,\lambda q,\lambda q/ef;q)_{\infty}}\\
&&\xqdn\:\:\times\:
 {_8\phi_7}\ffnk{ccccccc}
 {q;\frac{aq}{ef}}{\lambda,q\sqrt{\lambda},-q\sqrt{\lambda},\lambda b/a,\lambda c/a,\lambda d/a,e,f}
  {\sqrt{\lambda},-\sqrt{\lambda},aq/b,aq/c,aq/d,\lambda q/e,\lambda
  q/f},
 \enm
where $\lambda=qa^2/bcd$. Set $a=q$, $b=-q$,
$c=d=e=f=q^{\frac{1}{2}}$ to gain
 \bnm
&&\xqdn\sum_{n=0}^{\infty}\frac{1+q^{n+\frac{1}{2}}}{1+q^{\frac{1}{2}}}\frac{(1-q^{\frac{1}{2}})^3}{(1-q^{n+\frac{1}{2}})^3}(-q)^n
=\frac{(q,q^2,-q^{\frac{3}{2}},-q^{\frac{3}{2}};q)_{\infty}}{(-q,-q^2,q^{\frac{3}{2}},q^{\frac{3}{2}};q)_{\infty}}\\
&&\xqdn\:\:\times\:\sum_{k=0}^{\infty}\frac{1+q^{2k+1}}{1+q}\frac{(1+q^{\frac{1}{2}})^2}{(1+q^{k+\frac{1}{2}})^2}
\frac{(1-q^{\frac{1}{2}})^2}{(1-q^{k+\frac{1}{2}})^2}q^k.
 \enm
 Substituting $q^2$ for $q$ in the last equation, we have
\bnm
\sum_{k=0}^{\infty}\frac{1+q^{4k+2}}{(1+q^{2k+1})^2}\frac{q^{2k}}{(1-q^{2k+1})^2}
=\frac{(q;q^2)_{\infty}^2(-q^2;q^2)_{\infty}^2}{(-q;q^2)_{\infty}^2(q^2;q^2)_{\infty}^2}
\sum_{n=0}^{\infty}(-1)^n\frac{q^{2n}(1+q^{2n+1})}{(1-q^{2n+1})^3}.
 \enm
Calculating the series on the right hand side by Theorem
\ref{thm-a}, we achieve Theorem \ref{thm-b}.
\end{proof}

\begin{thm}\label{thm-c}
\bnm
\sum_{k=0}^{\infty}\frac{(q;q)_k}{(q;q^2)_{k+1}}q^{\binm{k+1}{2}}
=\frac{(q^2;q^2)_{\infty}^2}{(q;q^2)_{\infty}^2}.
 \enm
\end{thm}

\begin{proof}
We start with the summation formula for ${_2\phi_2}$-series (cf.
\citu{gasper}{p. 355})
 \bnm
  {_2\phi_2}\ffnk{ccccccc}{q;-q}{a^2, b^2}{abq^{\frac{1}{2}},-abq^{\frac{1}{2}}}
=\frac{(a^2q,b^2q;q^2)_{\infty}}{(q,a^2b^2q;q^2)_{\infty}}.
 \enm
Take $a=b=q^{\frac{1}{2}}$ to attain
 \bnm
\sum_{k=0}^{\infty}\frac{(q;q)_k}{(q^{\frac{3}{2}};q)_k(-q^{\frac{3}{2}};q)_k}q^{\binm{k+1}{2}}
=\frac{(q^2;q^2)_{\infty}^2}{(q,q^3;q^2)_{\infty}}.
 \enm
Multiplying both sides by $1/(1-q)$, we obtain Theorem \ref{thm-c}.
\end{proof}

\begin{thm}\label{thm-d}
\bnm
\quad\sum_{k=0}^{\infty}\frac{(q;q)_k^2}{(q;q)_{2k+1}}q^{2\binm{k+1}{2}}
=\frac{(q^3;q^3)_{\infty}^2}{(q, q^2;q^3)_{\infty}}.
 \enm
\end{thm}

\begin{proof}

Gasper and Rahman's identity (cf. \citu{gasper}{p. 110}) can be
expressed as
  \bnm
  &&\xxqdn\sum_{k=0}^{\infty}\frac{1-a^2q^{4k}}{1-a^2}\frac{(b,q^2/b;q)_k(a^2/q;q)_{2k}(c^3,a^2q^2/c^3;q^3)_k}
  {(a^2q^3/b,a^2bq;q^3)_k(q^2;q)_{2k}(a^2q/c^3,c^3/q;q)_k}q^k\\
 &&\xxqdn\:\:=\:\frac{(bq^2,q^4/b,bc^3/q,c^3q/b,c^3/a^2,c^3q^2/a^2,a^2q,a^2q^3;q^3)_{\infty}}
 {(q^2,q^4,c^3q,bc^3/a^2,a^2q^3/b,a^2bq,c^3q^2/a^2b;q^3)_{\infty}}\\
&&\xxqdn\:\:-\:\frac{(b,bq,bq^2,q^2/b,q^3/b,q^4/b,a^2/q,a^2q,a^2q^3,c^3/a^2,c^6q/a^2;q^3)_{\infty}}
 {(q^2,q^4,c^3/q,c^3q,a^2/c^3,a^2q/c^3,c^3q^3/a^2,c^3q^3/a^2b,a^3q^3/b,a^2bq,bq^3/a^2;q^3)_{\infty}}\\
&&\xxqdn\:\:\times\:\:
 {_2\phi_1}\ffnk{ccccccc}{q^3;q^3}{bc^3/a^2,c^3q^2/a^2b}{c^6q/a^2}.
 \enm
The terminating form of it due to Chu \citu{chu}{p. 65} is
  \bnm
  &&\xxqdn\sum_{k=0}^{n}\frac{1-a^2q^{4k}}{1-a^2}\frac{(b,q^2/b;q)_k(a^2/q;q)_{2k}(a^2q^{2+3n},q^{-3n};q^3)_k}
  {(a^2q^3/b,a^2bq;q^3)_k(q^2;q)_{2k}(a^2q^{1+3n},q^{-3n-1};q)_k}q^k\\
 &&\xxqdn\:\:=\:\frac{(a^2q;q)_{3n}(q^3,bq^2,q^4/b;q^3)_n}
 {(q^2;q)_{3n}(a^2q^2,a^2q^3/b,a^2bq;q^3)_n}.
 \enm
Let $n\to \infty$ to get
 \bnm
  &&\xxqdn\sum_{k=0}^{\infty}\frac{1-a^2q^{4k}}{1-a^2}\frac{(b,q^2/b;q)_k(a^2/q;q)_{2k}}
  {(a^2q^3/b,a^2bq;q^3)_k(q^2;q)_{2k}}q^{k^2+k}\\
 &&\xxqdn\:\:=\:\frac{(a^2q;q)_{\infty}(q^3,bq^2,q^4/b;q^3)_{\infty}}
 {(q^2;q)_{\infty}(a^2q^2,a^2q^3/b,a^2bq;q^3)_{\infty}}.
 \enm
 Fixing $a=0$, $b=q$ in the last
equation, we gain Theorem \ref{thm-d}.
\end{proof}

\begin{thm}\label{thm-e}
\bnm
\quad\sum_{k=0}^{\infty}\frac{1-q^{3k+2}}{1-q^2}\frac{(q^2;q^2)_k(q;q)_k^2}{(q^3;q^2)_{k}^3}q^{\binm{k+1}{2}}
=\frac{(q^4;q^2)_{\infty}(q^2;q^2)_{\infty}^3}{(q;q^2)_{\infty}(q^3;q^2)_{\infty}^3}.
 \enm
\end{thm}

\begin{proof}

Recall Chu's identity (cf. \citu{chu}{p. 63})
  \bnm
  &&\xxqdn\sum_{k=0}^{n}\frac{1-aq^{\frac{3k-1}{2}}}{1-aq^{-\frac{1}{2}}}
  \frac{(q^{-n},q^{n}a,q^{-\frac{1}{2}}a;q)_k(q^{-\frac{1}{2}}u,q^{-\frac{1}{2}}v,qa/uv;q^{\frac{1}{2}})_k}
  {(qa/u,qa/v,q^{-\frac{1}{2}}uv;q)_k(q^{\frac{1}{2}},aq^{n},q^{-n};q^{\frac{1}{2}})_k}q^{\frac{k}{2}}\\
 &&\xxqdn\:\:=\:\frac{(u,v,aq^{\frac{1}{2}},aq^{\frac{3}{2}}/uv;q)_n}
 {(q^{\frac{1}{2}},qa/u,qa/v,q^{-\frac{1}{2}}uv;q)_n}.
 \enm
The case $n\to \infty$ of it reads
 \bnm
  &&\qdn\xxqdn\sum_{k=0}^{\infty}\frac{1-aq^{\frac{3k-1}{2}}}{1-aq^{-\frac{1}{2}}}
  \frac{(q^{-\frac{1}{2}}a;q)_k(q^{-\frac{1}{2}}u,q^{-\frac{1}{2}}v,qa/uv;q^{\frac{1}{2}})_k}
  {(qa/u,qa/v,q^{-\frac{1}{2}}uv;q)_k(q^{\frac{1}{2}};q^{\frac{1}{2}})_k}q^{\frac{k^2+k}{4}}\\
 &&\qdn\xxqdn\:\:=\:\frac{(u,v,aq^{\frac{1}{2}},aq^{\frac{3}{2}}/uv;q)_{\infty}}
 {(q^{\frac{1}{2}},qa/u,qa/v,q^{-\frac{1}{2}}uv;q)_{\infty}}.
 \enm
  Set $a=q^{\frac{3}{2}}$, $u=v=q$ to achieve
 \bnm
\quad\sum_{k=0}^{\infty}\frac{1-q^{\frac{3k+2}{2}}}{1-q}\frac{(q;q)_k(q^{\frac{1}{2}};q^{\frac{1}{2}})_k^2}{(q^{\frac{3}{2}};q)_{k}^3}q^{\frac{k^2+k}{4}}
=\frac{(q^2;q)_{\infty}(q;q)_{\infty}^3}{(q^{\frac{1}{2}};q)_{\infty}(q^{\frac{3}{2}};q)_{\infty}^3}.
 \enm
Replacing $q$ by $q^2$ in the last equation, we attain Theorem
 \ref{thm-e}.
\end{proof}
\section{Telescoping method and $q$-analogues of $\pi$-formulas}
For a complex sequence $\{\tau_k\}_{k\in \mathbb{Z}}$, define
 the difference operator to be
\[\nabla\tau_k=\tau_k-\tau_{k-1}.\]
Then the telescoping method can offer the following relation:
 \bmn\label{telescoping-summation}
\:\qquad\sum_{k=0}^{n}\nabla\tau_k=\tau_n-\tau_{-1}.
 \emn
According to \eqref{telescoping-summation}, we shall establish a
general identity in this section. Subsequently, it will be utilized
to deduce $q$-analogues of \eqref{wei-a}-\eqref{guillera-b}.

\begin{thm}\label{thm-aa}
 Let $\{x_i\}_{i=1}^s$ and $\{y_i\}_{i=1}^s$ be complex
numbers. Then
 \bnm
&&\sum_{k=0}^{\infty}\frac{(x_1,x_2,\cdots,x_s;q)_k}{(qy_1,qy_2,\cdots,qy_s;q)_k}
\bigg\{\prod_{i=1}^s(1-q^kx_i)-\prod_{i=1}^s(1-q^ky_i)\bigg\}\\
&&\:\,=\:\frac{(x_1,x_2,\cdots,x_s;q)_{\infty}}{(qy_1,qy_2,\cdots,qy_s;q)_{\infty}}-\prod_{i=1}^s(1-y_i).
 \enm
\end{thm}

\begin{proof}
Taking
\[\tau_k=\frac{(qx_1,qx_2,\cdots,qx_s;q)_k}{(qy_1,qy_2,\cdots,qy_s;q)_k},\]
we obtain
 \bnm
\nabla\tau_k&&\xqdn=\:\frac{(qx_1,qx_2,\cdots,qx_s;q)_k}{(qy_1,qy_2,\cdots,qy_s;q)_k}
-\frac{(qx_1,qx_2,\cdots,qx_s;q)_{k-1}}{(qy_1,qy_2,\cdots,qy_s;q)_{k-1}}\\
 &&\xqdn=\:\frac{(x_1,x_2,\cdots,x_s;q)_k}{(qy_1,qy_2,\cdots,qy_s;q)_k}
\bigg\{\prod_{i=1}^s(1-q^kx_i)-\prod_{i=1}^s(1-q^ky_i)\bigg\}\prod_{i=1}^s(1-x_i)^{-1}.
 \enm
Substituting the expressions of $\tau_k$ and $\nabla\tau_k$ into
\eqref{telescoping-summation}, the result can be written as
  \bmn\label{terminating-sum}
&&\sum_{k=0}^{n}\frac{(x_1,x_2,\cdots,x_s;q)_k}{(qy_1,qy_2,\cdots,qy_s;q)_k}
\bigg\{\prod_{i=1}^s(1-q^kx_i)-\prod_{i=1}^s(1-q^ky_i)\bigg\}
 \nnm\\
&&\:\,=\:\frac{(x_1,x_2,\cdots,x_s;q)_{n+1}}{(qy_1,qy_2,\cdots,qy_s;q)_{n}}-\prod_{i=1}^s(1-y_i).
 \emn
Notice that
 \bnm
 &&\prod_{i=1}^s(1-q^kx_i)-\prod_{i=1}^s(1-q^ky_i)
 \\&&\:=\big\{1+a_1q^k+a_2q^{2k}+\cdots+a_sq^{sk}\big\}
\\&&\:-\:\big\{1+b_1q^k+b_2q^{2k}+\cdots+b_sq^{sk}\big\}
\\&&\:=(a_1-b_1)q^k+(a_2-b_2)q^{2k}+\cdots+(a_s-b_s)q^{sk}.
\enm
 The case $n\to \infty$ of \eqref{terminating-sum} reads
\bnm
&&\sum_{k=0}^{\infty}\frac{(x_1,x_2,\cdots,x_s;q)_k}{(qy_1,qy_2,\cdots,qy_s;q)_k}
\bigg\{\prod_{i=1}^s(1-q^kx_i)-\prod_{i=1}^s(1-q^ky_i)\bigg\}\\
&&\:\,=\:\frac{(x_1,x_2,\cdots,x_s;q)_{\infty}}{(qy_1,qy_2,\cdots,qy_s;q)_{\infty}}-\prod_{i=1}^s(1-y_i).
 \enm
\end{proof}

 Performing the the replacements  $y_i\to1$, $x_{m+i}\to q/x_i$, $y_{m+i}\to q$ with $1\leq i\leq m$ in Theorem
\ref{thm-aa} when $s=2m$, we get the following identity.

\begin{corl}\label{corl-aa}
 Let $\{x_i\}_{i=1}^m$ be complex numbers. Then
 \bnm
 \frac{\prod_{i=1}^m(x_i,q/x_i;q)_{\infty}}{(q,q^2;q)_{\infty}^m}&&\xqdn=\:
  \sum_{k=0}^{\infty}\frac{\prod_{i=1}^m(x_i,q/x_i;q)_{k}}{(q,q^2;q)_k^m}\\
  &&\xqdn\times\:\bigg\{\prod_{i=1}^m(1-q^kx_i)(1-q^{k+1}/x_i)-(1-q^k)^m(1-q^{k+1})^m\bigg\}.
 \enm
\end{corl}

Substituting $q^{x_i}$ for $x_i$ and then letting $q\to1$, Corollary
\ref{corl-aa} becomes
 \bnm
 \frac{\prod_{i=1}^m\sin(\pi x_i)}{\pi^m}&&\xqdn=\:
  \sum_{k=0}^{\infty}\frac{\prod_{i=1}^m(x_i)_{k}(1-x_i)_k}{k!^m(k+1)!^m}\bigg\{\prod_{i=1}^m(k+x_i)(k+1-x_i)-k^m(k+1)^m\bigg\}.
 \enm

When $m=2$, Corollary \ref{corl-aa} reduces to the following
$q$-analogue of \eqref{wei-a}:
 \bnm
 &&\frac{(x,y,q/x,q/y;q)_{\infty}}{(q,q^2;q)_{\infty}^2}=
  \sum_{k=0}^{\infty}\frac{(x,y,q/x,q/y;q)_{k}}{(q,q^2;q)_k^2}\\
  &&\times\:\Big\{(1-q^kx)(1-q^ky)(1-q^{k+1}/x)(1-q^{k+1}/y)-(1-q^k)^2(1-q^{k+1})^2\Big\}.
 \enm

When $x=y=q^{1/2}$, the last equation reduces to the following
$q$-analogue of \eqref{guillera-b}:
  \bnm
 &&\frac{(q^{\frac{1}{2}};q)_{\infty}^4}{(q,q^2;q)_{\infty}^2}=
  \sum_{k=0}^{\infty}\frac{(q^{\frac{1}{2}};q)_{k}^4}{(q,q^2;q)_k^2}
  \Big\{(1-q^{k+\frac{1}{2}})^4-(1-q^k)^2(1-q^{k+1})^2\Big\}.
 \enm

Performing the the replacements $x_i\to q$, $y_i\to x_i$,
$x_{m+i}\to q$, $y_{m+i}\to q^2/x_i$ with $1\leq i\leq m$ in Theorem
\ref{thm-aa} when $s=2m$, we gain the following identity.

\begin{corl}\label{corl-bb}
 Let $\{x_i\}_{i=1}^m$ be complex numbers. Then
 \bnm
 \frac{(q,;q)_{\infty}^{2m}}{\prod_{i=1}^m(x_i,q/x_i;q)_{\infty}}&&\xqdn=\:
  \frac{1}{\prod_{i=1}^m(1-q/x_i)}+\sum_{k=0}^{\infty}\frac{(q;q)_k^{2m}}{\prod_{i=1}^m(x_i;q)_{k+1}(q/x_i;q)_{k+2}}\\
  &&\xqdn\times\:\bigg\{(1-q^{k+1})^{2m}-\prod_{i=1}^m(1-q^kx_i)(1-q^{k+2}/x_i)\bigg\}.
 \enm
\end{corl}

Substituting $q^{x_i}$ for $x_i$ and then letting $q\to1$, Corollary
\ref{corl-bb} becomes
 \bnm
 \quad\frac{\pi^m}{\prod_{i=1}^m\sin(\pi x_i)}&&\xqdn=\:\frac{1}{\prod_{i=1}^m(1-x_i)}+
  \sum_{k=0}^{\infty}\frac{k!^{2m}}{\prod_{i=1}^m(x_i)_{k+1}(1-x_i)_{k+2}}\\
  &&\xqdn\times\:\Big\{(k+1)^{2m}-\prod_{i=1}^m(k+x_i)(k+2-x_i)\Big\}.
 \enm

When $m=2$, Corollary \ref{corl-bb} reduces to the following
$q$-analogue of \eqref{wei-b}:
 \bnm
 &&\frac{(q;q)_{\infty}^4}{(x,y,q/x,q/y;q)_{\infty}}=\frac{1}{(1-q/x)(1-q/y)}+
  \sum_{k=0}^{\infty}\frac{(q;q)_k^4}{(x,y;q)_{k+1}(q/x,q/y;q)_{k+2}}\\
  &&\:\times\:\Big\{(1-q^{k+1})^4-(1-q^kx)(1-q^ky)(1-q^{k+2}/x)(1-q^{k+2}/y)\Big\}.
 \enm

\textbf{Acknowledgments}

 The work is supported by the National Natural Science Foundation of China (No. 11661032).


\end{document}